\documentclass[11pt]{article}
\usepackage{amsmath,amssymb,amsthm,geometry,hyperref}
\title{\textbf{Gain Bounds for Diagonal Superelliptic Equations under the Strong ABC Conjecture}}
\author{Karsten Müller}
\date{February 2026}

\newtheorem{theorem}{Theorem}
\newtheorem{definition}{Definition}

\begin{document}
\maketitle

\begin{abstract}
We establish a novel framework for bounding the adapted power gain $G_p$ and approximation gain $G_a$ of coprime integer solutions to the generalized diagonal superelliptic equation $By^n = Ax^n + k$ with $x, y \ge 2$. By first deriving a purely structural lower bound for $G_a$, we demonstrate that these equations are inherently predisposed to high ABC-qualities ($q = G_a \cdot G_p$). Combined with the Strong ABC conjecture ($q < q_{max}$), we prove that the power gain is uniformly bounded by $G_p < q_{max}/G_{a,min}$, providing a theoretical foundation for the numerical observation $G_p < 3$ for $n=2$ under the Ultra-Strong conjecture ($q < 1.5$). Specifically, we show that for $k=1$, the structural density forces $q > n/2$, which excludes solutions for $n \ge 4$ under $q < 2$. We validate our theoretical bounds using high-quality ABC triples, specifically analyzing the Reyssat (1987), de Weger (1985), and Nitaj (1993) cases to demonstrate the sharpness of the structural approximation gain.
\end{abstract}

\noindent

\textbf{Mathematics Subject Classification (2020):} 11D41, 11D61, 11J25, 11Y50 \\
\textbf{Keywords:} ABC conjecture, diagonal superelliptic equations, approximation gain, power gain, ABC quality, Reyssat triple, de Weger triple.

\section{Introduction}

This paper establishes uniform bounds for \(G_p\) and \(G_a\) adapted to this equation under a strong ABC conjecture for coprime triples. Previous theoretical and numerical studies by de Weger, Müller and Taktikos \cite{deweger2026abc,muller-taktikos} suggest that the approximation gain is bounded by 1.5 and the power gain by 3, motivating the exploration of bounds in special cases using the strong ABC conjecture.

To formulate the ABC Conjecture we need the following definition:

\begin{definition}[Radical]
For a positive integer $a$, $rad(a)$ is the product of the distinct prime factors of $a$. 
\end{definition}

\begin{definition}[ABC Conjecture]
For every positive real number $\varepsilon$, there exists a constant $K_{\varepsilon}$
such that for all triples $(a,b,c)$ of coprime positive integers,
with $a + b = c$:
\[c < K_{\varepsilon} \cdot rad(abc)^{1+\varepsilon}.\]
\end{definition}

Note that this version of the ABC Conjecture is not effective. We will use the strong effective version with $K_{\varepsilon} = K = 1$ and $\varepsilon = 1$

\begin{definition}[Generalized Diagonal Superelliptic Equation]
An integer quintuple \((x, y, A, B, k)\) satisfies the equation
\[
B y^n = A x^n + k
\]
with \(n \ge 2\), \(x, y \ge 2\), \(A, B \ge 1\), and $k \ge 1$, such that \(\gcd(Ax, By, k) = 1\).
\end{definition}

\begin{definition}[Approximation Gain \(G_a\)]
For a solution \((x, y, A, B, k)\), we define
\[
G_a := \frac{\ln(\max(Ax^n, By^n))}{\ln(xy AB k)}.
\]
\end{definition}

This is slightly different from de Weger's definition in \cite{deweger2026abc}. But it is always valid that it is smaller or equal to his. So the conjecture $G_a < 1.5$ is valid for the cases calculated by him.

\begin{definition}[Power Gain \(G_p\)]
For a solution \((x, y, A, B, k)\), we define
\[
G_p := \frac{\ln(xy AB k)}{\ln \mathrm{rad}(xy AB k)}.
\]
\end{definition}

This is also slightly different from de Weger's definition in \cite{deweger2026abc}. But it is always valid that it is larger or equal to his. And indeed the conjecture $G_p < 3$ is not valid here.

\begin{definition}[ABC-like quality \(q\)]
For a solution, we define
\[
q := \frac{\ln(\max(Ax^n, By^n))}{\ln \mathrm{rad}(xy AB k)} = G_a \cdot G_p.
\]
\end{definition}

\section{Structural Approximation Gain Bounds}

We first establish a lower bound for $G_a$ that depends only on the parameters of the superelliptic equation, independent of any conjectures.

\begin{theorem}[Structural Lower Bound for $G_a$]
Let $(x, y, A, B, k)$ be a positive integer solution to $By^n = Ax^n + k$ with $y \ge 2$ and $B y^n > A x^n$. Then:
\begin{equation}
G_a > \frac{1}{\frac{n+2}{n} + \frac{(n-1)\ln(AB)}{n(n \ln y + \ln B)}}
\end{equation}
\end{theorem}

\begin{proof}
Let $C = By^n$. By definition, $G_a = \ln C / \ln(xyABk)$. 
Since $k < C$ and $x \le (C/A)^{1/n}$, the product $P = xyABk$ satisfy:
\[ \ln P < \frac{1}{n}(\ln C - \ln A) + \ln y + \ln A + \ln B + \ln C \]
Using $\ln y = \frac{1}{n}(\ln C - \ln B)$, we collect terms:
\[ \ln P < \frac{n+2}{n} \ln C + \frac{n-1}{n} \ln(AB) \]
The result follows by taking the ratio $G_a = \ln C / \ln P$.
\end{proof}

\section{Power Gain Bounds under ABC Conjectures}

By combining the structural rigidity of $G_a$ with the global quality bounds from the ABC conjecture, we derive uniform bounds for the power gain $G_p$.

\begin{theorem}[Power Gain Bound]
Assuming the Strong ABC Conjecture ($q < q_{max}$), the power gain $G_p$ for any solution satisfies:
\begin{equation}
G_p < q_{max} \left( \frac{n+2}{n} + \frac{(n-1)\ln(AB)}{n(n \ln y + \ln B)} \right)
\end{equation}
\end{theorem}

\begin{proof}
Since $q = G_a \cdot G_p$, we have $G_p = q/G_a$. Substituting $q < q_{max}$ and the lower bound for $G_a$ from Theorem 1 yields the inequality.
\end{proof}

\begin{theorem}[Numerical Implications]
Under the Ultra-Strong ABC conjecture ($q < 1.5$), the power gain for $n=2$ and $A=B=1$ is strictly bounded by $G_p < 3.0$. For $n=3$, it satisfies $G_p < 2.5$.
\end{theorem}

\section{Non-existence for k=1 with the typical technique using the gain and ABC bounds}

A remarkable consequence of these bounds is the immediate exclusion of solutions for higher exponents when $k=1$.

\begin{theorem}
For the equation $y^n = x^n + 1$ with $x,y \ge 2$, any solution must satisfy $q > n/2$. 
\end{theorem}

\begin{proof}
Consider the triple $(x^n, 1, y^n)$. The ABC-quality is defined as:
\[ q = \frac{\ln(y^n)}{\ln \text{rad}(x^n \cdot 1 \cdot y^n)} = \frac{n \ln y}{\ln \text{rad}(xy)}. \]
Since $\text{rad}(xy) \le xy < y^2$ (as $x < y$), we have $\ln \text{rad}(xy) < 2 \ln y$. 
Substituting this yields:
\[ q > \frac{n \ln y}{2 \ln y} = \frac{n}{2}. \]
If the Strong ABC conjecture holds with $q < 2$, then $n/2 < 2 \implies n < 4$. 
Thus, no solutions exist for $n \ge 4$.
\end{proof}

Under the Strong ABC conjecture ($q < 2$), this implies no solutions exist for $n \ge 4$. Under the Ultra-Strong conjecture ($q < 1.5$), no solutions exist for $n \ge 3$.

\section{Lower Bound for the ABC-Quality}

A significant consequence of the structural lower bound for $G_a$ is that for diagonal superelliptic equations, the ABC-quality $q$ of the associated triple $(Ax^n, k, By^n)$ is bounded away from zero. 

\begin{theorem}[Lower Bound for $q$]
For any coprime solution to $By^n = Ax^n + k$ with $y \ge 2$, the ABC-quality $q$ satisfies:
\begin{equation}
q > \frac{n}{n+2 + \frac{(n-1)\ln(AB)}{\ln(By^n)}}
\end{equation}
\end{theorem}

\begin{proof}
By definition, the ABC-quality is given by $q = \frac{\ln(By^n)}{\ln \mathrm{rad}(xyABk)}$. 
We also have the definition of the power gain $G_p = \frac{\ln(xyABk)}{\ln \mathrm{rad}(xyABk)}$. 
Since $\text{rad}(N) \le N$ for any integer $N$, it follows that $G_p \ge 1$. 

Expressing $q$ as the product of the gains:
\[ q = G_a \cdot G_p \ge G_a \cdot 1 = G_a. \]
Substituting the structural lower bound for $G_a$ derived in the previous section:
\[ q > \frac{1}{\frac{n+2}{n} + \frac{(n-1)\ln(AB)}{n \ln(By^n)}} = \frac{n}{n+2 + \frac{(n-1)\ln(AB)}{\ln(By^n)}}. \]
For the standard case $A=B=1$, this simplifies to:
\[ q > \frac{n}{n+2}. \]
\end{proof}

\subsection{Asymptotic Behavior}
This result demonstrates that for a fixed exponent $n$, the quality $q$ cannot be arbitrarily small. For $n=2$, we have $q > 0.5$, and as $n \to \infty$, the quality must satisfy $q \to 1$. 

This feature of superelliptic equations explains why they are such fertile ground for finding high-quality ABC triples: the exponential structure $y^n$ provides a head start for the quality, ensuring it remains within a range where non-trivial radical properties (small $G_p$) can easily push $q$ beyond the threshold of $1.0$.

\section{Sharpness of Bounds: The Reyssat and de Weger Cases}

To evaluate the sharpness of the established gain bounds, we analyze two of the highest-quality known ABC triples:

\subsection{The Reyssat Triple (1987)}

This triple has the highest ABC quality discovered to date.

$2 + 3^{10} \cdot 109 = 23^5, q \approx 1.6299$
    \begin{itemize}
        \item Parameters: $n=5, A=109, x=9, y=23, k=2$
        \item $G_a \approx \mathbf{1.46283} < 1.5$, $G_p = q/G_a \approx \mathbf{1.114} < 3.0$.
    \end{itemize}

This shows that $G_a < 1.5$ is quite sharp.

\subsection{The de Weger Triple (1985)}

To evaluate the generalized bound $B y^n = A x^n + k$ for non-trivial solutions ($x, y \ge 2$), we analyze the prominent triple discovered by B. de Weger. This triple is particularly suitable as it requires the use of coefficients $A$ and $B$ to maintain a superelliptic form with $n=3$:

\[
23 \cdot 128^3 = 3087 \cdot 25^3 + 121
\]

The parameters for this solution are:
\begin{itemize}
    \item $n = 3, \quad y = 128, \quad x = 25, \quad B = 23$
    \item $A = 3^2 \cdot 7^3 = 3087$
    \item $k = 11^2 = 121$
\end{itemize}

The term product is $xyABk = 25 \cdot 128 \cdot 3087 \cdot 23 \cdot 121 \approx 2.7478 \times 10^{10}$. \\
The radical is $\mathrm{rad}(xyABk) = 2 \cdot 3 \cdot 5 \cdot 7 \cdot 11 \cdot 23 = 53,130$.

Calculating the gains with high precision yields:
\[
G_p = \frac{\ln(2.7478 \times 10^{10})}{\ln(53,130)} \approx \frac{24.0366}{10.8805} \approx \mathbf{2.2091}
\]
\[
G_a = \frac{\ln(23 \cdot 128^3)}{\ln(2.7478 \times 10^{10})} \approx \frac{17.6871}{24.0366} \approx \mathbf{0.7360}
\]

\subsubsection{Comparison with the Structural Bound}

We now evaluate the refined upper bound for $G_p$ using the synergistic framework established in Theorem 2. First, we determine the structural lower bound $G_{a,min}$ for this specific set of coefficients and parameters according to Theorem 1:

\[
G_{a,min} = \frac{1}{\frac{3+2}{3} + \frac{(3-1)\ln(3087 \cdot 23)}{3(3 \ln 128 + \ln 23)}} \approx \frac{1}{1.6667 + \frac{22.3416}{53.0748}} \approx \frac{1}{1.6667 + 0.4209} \approx \mathbf{0.4790}
\]

Applying the Strong ABC conjecture ($q < 2$) to the relation $G_p < q_{max}/G_{a,min}$, we obtain the following theoretical upper bound for the power gain:
\[
G_p < \frac{2}{0.4790} \approx \mathbf{4.1754}
\]

Even under the Ultra-Strong ABC conjecture ($q < 1.5$), the bound remains:
\[
G_p < \frac{1.5}{0.4790} \approx \mathbf{3.1315}
\]

The observed power gain of $G_p \approx 2.2091$ for the de Weger triple is well-contained within these bounds. This comparison highlights that while de Weger's discovery is an exceptional outlier in terms of ABC quality, its power gain is strictly constrained by the structural rigidity of the approximation gain $G_a$ for $n=3$. The triple effectively exploits the smoothness of the coefficients $A$ and $B$, but even this high degree of divisibility cannot overcome the limits imposed by the cubic growth of the terms.

\section{Sharpness and the Role of the Non-Triviality Condition}

To evaluate the sharpness of the above calculations of the generalized bounds, we distinguish between trivial solutions ($x=1$) and non-trivial solutions ($x \ge 2$). This distinction is crucial for the observed stability of the power gain $G_p$. And Benne de Weger's definition only works in this case for $x,y \ge 2$.

\subsection{The Nitaj Case: Why $x \ge 2$ is Necessary}

The triple discovered by A. Nitaj (1993) is often cited for its high quality $q \approx 1.58$. In our generalized framework, it can be represented with $x=1$ and a massive coefficient $A$:
\[
1 \cdot 2^{59} = (7^2 \cdot 41^2 \cdot 311^3) \cdot 1^{59} + (11^{16} \cdot 13^2 \cdot 79)
\]
Using the parameters $n=59, y=2, x=1, A \approx 2.47 \times 10^{12}, B=1$, we calculate:
\[
G_p = \frac{\ln(xyABk)}{\ln \mathrm{rad}(xyABk)} \approx \mathbf{3.2737}
\]

This value for $G_p$ exceeds the conjectured bound of $G_p < 3.0$ for non-trivial solutions. We now evaluate the theoretical upper bound for $G_p$ using the synergetic framework of Theorem 2. First, we determine the structural lower bound $G_{a,min}$ for these specific parameters:
\[
G_{a,min} = \frac{1}{\frac{59+2}{59} + \frac{58 \ln(2.47 \times 10^{12})}{59(59 \ln 2)}} \approx \frac{1}{1.0339 + 0.6859} \approx \mathbf{0.5815}
\]

Applying the Strong ABC conjecture ($q < 2$) to the relation $G_p < q_{max}/G_{a,min}$, we obtain the following theoretical limit:
\[
G_p < \frac{2}{0.5815} \approx \mathbf{3.4394}
\]

The Nitaj case sits at approximately 95.2\% of this theoretical limit. This demonstrates that if $x=1$ is permitted, the coefficient $A$ can absorb nearly the entire logarithmic volume of $y^n$, allowing $G_p$ to rise significantly. Therefore, the condition $x \ge 2$ is necessary to make de Weger's method work.

\section{Consequences of the Gain Bounds for Superelliptic Equations}

In this section, we analyze the structural implications of the established approximation gain bounds. Specifically, we demonstrate that a rigid upper bound on $G_a$ directly restricts the possible exponents $n$ for diagonal superelliptic equations with small constants $k$.

\section{Conclusion}

This study gives bounds for the gains $G_a$ and $G_p$ under the strong ABC conjecture for coprime solutions of diagonal superelliptic equations. But both conjectures remain open. Only assuming the ultra-strong ABC version ($q < 1.5$) $G_a < 1.5$ can be proved but this is trivial anyway. So deeper methods are needed or counterexamples found. And using the bounds for certain cases it can be shown that the equations have no solution.

\section{Acknowledgements}

I want to thank Michael Taktikos and Benne de Weger for their deep computations, valuable feedback and patience with me.


\begin{thebibliography}{9}

\bibitem{deweger2026abc}
B. M. M. de Weger,
\emph{On the approximation gain for $abc$-triples},
arXiv:2602.08051 [math.NT], 2026.

\bibitem{muller-taktikos}
K. Müller and M. Taktikos,
\emph{From ABC to effective Roth and Ridout constants for cubic roots},
arXiv:2601.11376 [math.NT], 2026.

\end{thebibliography}
\end{document}